\documentclass[11pt]{article}
\usepackage{amsmath,empheq} % \font used for R in Real numbers
\usepackage{verbatim}
\usepackage{amssymb, float}
\usepackage{color}
\usepackage{graphicx,epsfig}
\usepackage[applemac]{inputenc}
\title{An implicit sweeping process approach to quasistatic evolution variational inequalities.}
\newtheorem{theorem}{Theorem}[section]
\newtheorem{corollary}[theorem]{Corollary}
\newtheorem{lemma}[theorem]{Lemma}
\newtheorem{proposition}[theorem]{Proposition}

\newtheorem{example}[theorem]{Example}
\newtheorem{remark}[theorem]{Remark}
\def\beq{\begin{equation}}
\def\eeq{\end{equation}}
\def\baq{\begin{eqnarray}}
\def\eaq{\end{eqnarray}}
\def\baqn{\begin{eqnarray*}}
\def\eaqn{\end{eqnarray*}}

\newcommand{\R}{\mathbb{R}}
\newcommand{\N}{\mathbb{N}}

\def\image #1 (#2,#3) (echelle #4) #5{
\dimen2=#2
\dimen3=#3
\divide \dimen2 by 1000
\multiply \dimen2 by #4
\divide \dimen3 by 1000
\multiply \dimen3 by #4
\setbox1 =\vbox to \dimen2{\hsize=\dimen3\vfill\special{picture #1
scaled #4}}
\vbox{\hsize=\dimen3\box1\medskip\centerline{#5}}
}
\begin{document}
\author{
Samir {\sc Adly}\thanks{XLIM UMR-CNRS 7252, Universit\'e de Limoges, 123, avenue Albert Thomas, 87060 Limoges CEDEX, France.
Email: samir.adly@unilim.fr}  \;and Tahar {\sc Haddad}\thanks{Laboratoire LMPEA, Universit\'{e} de Jijel, B.P. 98, Jijel 18000, Alg\'{e}rie.Email: haddadtr2000@yahoo.fr}}

\maketitle
\begin{abstract}\noindent In this paper, we study a new variant of Moreau's sweeping process with
velocity constraint. Based on an adapted version of Moreau's catching-up algorithm, we show the well-posedness (in the sense existence and
uniqueness) of this problem in a general framework. We show the equivalence between this implicit sweeping process and a quasistatic evolution variational inequality.
It is well known that the variational
formulations of many mechanical problems with unilateral contact and friction lead to an evolution
variational inequality. As an application, we reformulate the quasistatic antiplane frictional contact
problem for linear elastic materials with short memory as an implicit sweeping process with
velocity constraint. The link between the implicit sweeping process and the quasistatic
evolution variational inequality is possible thanks to some standard
tools from convex analysis and is new in the literature.
\end{abstract}

\noindent {\bf Keywords} Moreau's sweeping process, \and evolution variational inequalities, \and unilateral constraints, \and quasistatic frictional contact problems.\\
\noindent {\bf AMS subject classifications} 49J40, 47J20, 47J22,  34G25, 58E35, 74M15, 74M10, 74G25.
%\tableofcontents[hideallsubsections]
\setcounter{tocdepth}{1}
\tableofcontents

%%%%%%%%%%%%%%%%%%%%%%%%%%
\section{Introduction}
The notion of the so-called sweeping process was introduced by Jean Jacques Moreau in the 1970s. Jean Jacques Moreau wrote more than 25 papers devoted to the treatment of both theoretical and numerical aspects  of the sweeping process as well as its applications in unilateral mechanics \cite{M,Moreau1,Moreau2,Moreau3,Moreau4}. It was first considered for modeling the quasistatic evolution of elastoplastic systems. The sweeping process consists of finding a trajectory $t\in [0,T]\mapsto u(t)\in C(t)$ satisfying the following generalized Cauchy problem
\begin{equation}\label{eq1.1}
 -\dot{u}(t) \in {\rm N}_{C(t)}( u(t))\quad\mbox{ a.e. on }\;[0,T],  u(0)=u_{0}\in C(0),
\end{equation}
where $C:[0,T]\rightrightarrows H$ is a set-valued mapping defined
from $[0,T] \; (T>0)$ to a Hilbert space $H$ with convex
and closed values, and ${\rm N}_{C(t)}( u(t))$ denotes the outward normal
cone, in the sense of convex analysis, to the set $C(t)$ at {the point}
$u(t).$
 Translating inclusion (\ref{eq1.1}) to a mechanical language, we obtain the following interpretation:\\
- If the position $u(t)$ of a particule lies in the interior of the moving set $C(t)$, then {the normal cone is reduced to the singleton $\{0\}$ and hence }$\dot{u}(t)=0$, which means that the particule remains at rest.\\
- When the boundary of $C(t)$ catches up {with} the particle, then this latter is pushed in an inward normal direction by the boundary of $C(t)$ to stay inside $C(t)$ and satisfies the {viability} constraint {$u(t)\in C(t)$}. This mechanical visualization led Moreau to call this problem the {\em sweeping process}: the particle is swept by the moving set.\\
Using the definition of the normal cone, it is easy to see that  (\ref{eq1.1})  is equivalent to the following evolution variational inequality:
$$\left\{
\begin{array}{l}
\mbox{Find } u(t)\in C(t) \mbox{ such that }\\
\langle \dot u(t),v-u(t)\rangle \geq 0,\;\;\forall v\in C(t) \mbox{ and for a. e. } t\in[0,T].
\end{array}
\right.
$$
In nonsmooth mechanics, the moving set is usually expressed in inequalities form, corresponding to the so-called unilateral constraints,
\begin{equation}\label{ineq}
C(t):=\bigcap_{i=1}^m \big\{x\in H: f_i(t,x)\le 0\big\},
\end{equation}
where $f_i: [0,T]\times \R^n\to \R,$ ($i=1,2,\ldots,m$) are some given regular convex functions.
%Consequently the
%sweeping process includes as a special case the following
%evolution inequality.\\ Find $u(t)\in K$ for all $t\in [0,T] $
%such that
%$$
%\langle \dot u(t),v-u(t)\rangle \geq \langle f(t),v-u(t)\rangle,
%$$
%for a.e. $t\in[0,T]$ and for all $v\in K,$ where K is a closed
%convex subset of a Hilbert space H, $u:[0,T]\longrightarrow H$ and
%$f\in L^{2}([0,T],H).$\\

Several extensions of the sweeping
process in diverse ways have been studied in the literature (see e.g. \cite{AHT, KM2} and references therein). A natural generalization of the sweeping process is the differential inclusion
\begin{equation}\label{eq-intro2}
   \left\{
     \begin{array}{l}
       -\dot{u}(t)\in {\rm N}_{C(t)}(u(t))+f(t,u(t))+\mathcal{F}(t,u(t))\\
       u(0)=u_0\in C(0)\\
       u(t)\in C(t),\quad\forall t\in [0,T],
     \end{array}
   \right.
\end{equation}
where $f:[0,T]\times H\to H$ is a Lipschitz mapping and $\mathcal{F}$ is a set-valued mapping from $[0,T]\times H$ into weakly compact convex sets of a Hilbert space $H$.\\
In this paper we are interested in
a new variant of the sweeping process {of the following form}
\begin{equation}\label{eq1.2}
\left\{
\begin{array}{c}
 -\dot u(t) \in {\rm N}_{C(t)}\big(A\dot u(t)+Bu(t)\big) \quad\mbox{ a.e. }t\in [0,T].\\
  u(0)=u_{0}\in H.\hspace{30mm}
\end{array}
\right.
\end{equation}
We assume that the following assumptions hold:\\
\\ \textit{$(\mathcal {SP}_{1})$} $A,B: H\longrightarrow H$ are two linear, bounded and symmetric operators satisfying:
$$\left\{
\begin{array}{l}
\langle Ax,x\rangle\geq \beta \|x\|^{2} , \text {for all } x\in H
\text { for some constant } \beta > 0\\
\langle Bx,x\rangle\geq 0 , \text {for all } x\in H.
\end{array}
\right.
$$
\\ \textit{$(\mathcal {SP}_{2})$} For every $t\in [0,T]$,
$C(t) \subset H$ is a closed convex and nonempty set { such that} $t\mapsto C(t)$ is absolutely  continuous, in the sense that there
exists a nondecreasing
 {absolutely {continuous} function} {$v:[0,T]\to \R^+$} with $v(0)=0$ such
that {$$ d_{H}(C(t),C(s))\leq v(t)-v(s), \text
{for all } 0\leq s\leq t\leq T,
$$}
%{Question: $v$ nondecreasing in general? why you did not use it?}\\
where $d_{H}$ denotes the Hausdorff distance defined in (\ref{hausd}). \\
%The compatibility condition $0\in C(0)$ is easier to satisfy in
%applications ( see the motivation below). Also, in this example,
%the sets $C(t):=[t,+\infty[$ for $t\in [0,T]$ are closed and
%convex with $0\in C(0)$ and $C(.)$ is Lipschitz with constant
%$L=1$ w.r.t. to $d_{H}.$ Observe that $C(0)=[0,+\infty[$
%\textbf{is not necessarily bounded}.
%\tcb{We note that is possible to suppose that the function $v$ in (\ref{acont})  }
It is worth mentioning that in the particular case where $A:=I$ and $B:=0$, problem (\ref{eq1.2}) has been studied in \cite{bounkhel} by assuming that the set $C(t)$ is prox-regular and satisfying a compactness condition.\\
 The main goal of this paper is to prove a general
existence and uniqueness result for the implicit differential
inclusion described by (\ref{eq1.2}) by assuming that the set-valued mapping $t\mapsto C(t)$ moves in an absolutely continuous way with respect to the Hausdorff distance. By using an implicit time discretization, we solve at each iteration a variational inequality. The limit of a sequence of functions, constructed via linear interpolation, is showed to be a solution of
 (\ref{eq1.2}). For the particular case when the moving set $C(t)=C-f(t)$ (with $f\in W^{1,1}([0,T];H)$ and $C$ a fixed closed convex subset of $H$), we give an application to the quasistatic  frictional contact problem involving viscoelastic materials with short memory \cite{DL,SM}.
\\
The paper is organized as follows. In section {2}, we
introduce some notations and state some preliminary results which
will be used to establish the existence of discretization and to prove the convergence of the approximants. In section
 {3}, we present an existence and uniqueness theorem related
to the new variant of the sweeping process problem (\ref{eq1.2}).  In section 4, we give an application to the quasistatic frictional contact problem.
%The implicit evolution variational
%inequality of type (\ref{eq1.3}) is considered with some numerical
%experiments in section \textbf{4}.

\section{Notation and preliminaries}

Let $H$ be a real separable Hilbert  space endowed with the inner
product $\langle \cdot,\cdot\rangle$ and the associated norm
$\Vert\cdot\Vert$. For any $x\in H$ and $r\geq 0,$ the closed ball
centered at $x$ with radius $r$ will be denoted by
$\overline{\mathbb{B}}(x,r).$ For $x=0$ and $r=1,$ we will set
$\overline{\mathbb{B}}$ instead of $\overline{\mathbb{B}}(0,1).$
%For a closed convex subset $C$ of $H$ and $x\in C$, the normal cone to $C$ at $x$ is denoted by
%$$
%{\rm N}_{C}(x)=\{v\in H:\langle v, y-x\rangle \leq 0, \forall y\in C\}.
%$$
Given a set-valued map $A:H\rightrightarrows H$, we denote by $D(A)$, $G(A)$ and $R(A)$ respectively the domain, the graph and the range of $A$, defined by
\begin{eqnarray*}
&&D(A) = \{x\in H\;:\; A(x) \neq \emptyset\},\;G(A) = \{(x, y)\in
\textcolor{black}{D(A)}\times H\; :\; y \in A(x)\}\\
 &&\mbox{ and } R(A)=\bigcup_{x\in D(A)}
A(x) .
\end{eqnarray*}
We define the \textit{inverse} of $A$, $A^{-1}$ by
  $$ y\in A(x) \Longleftrightarrow x\in A^{-1}(y),\mbox{ i.e. } (x,y)\in G(A) \Longleftrightarrow (y,x)\in G(A^{-1}).$$
We say that $A:H\rightrightarrows H$ is \textit{monotone} iff
  $$\langle x^*-y^*,x-y\rangle\geq 0,\;\forall x^*\in A(x),\;\forall y^*\in A(y).   $$
We say that $A:H\rightrightarrows H$ is \textit{maximal monotone}
iff it is monotone and its graph is maximal in the sense of the
inclusion, i.e., $G(A)$ is not properly contained in the graph of
any other monotone operator.

 Let $J :H \rightarrow ]-\infty,+\infty]$ be a lower semicontinuous, convex and proper function, i.e. $J\in \Gamma_0(H)$. The effective domain of $J$, denoted by Dom$(J)$ is defined by
 $${\rm Dom}(J)=\{x\in H\;:\;J(x)<+\infty\}.$$
For any $x\in {\rm Dom}(J)$, the subdifferential of $J$ at $x$ is defined by
\begin{equation}
\label{subdiff}
\partial J (x)=\{\xi\in H:\langle\xi, y-x\rangle
\leq J(y)-J (x),\forall y\in H\}.
\end{equation}
We recall that for $x\not\in {\rm Dom} (J)$, $\partial J(x)=\emptyset$ and that if $J$ is of class $C^1$ at $x$, then
$\partial J(x)=\{\nabla J(x)\}.$\\
For the above function $J$, its {\em Legendre-Fenchel conjugate} is defined as
$$
  J^*:H \to \R\cup\{-\infty,+\infty\} \quad \text{with}\;J^*(x^*)
  :=\sup_{x\in H}\big(\langle x^*,x \rangle -J(x)\big).
$$
The Legendre-Fenchel conjugate is also related to the subdifferential. Indeed, for $J(x)$ finite,
one has
$$
   x^* \in\partial J(x) \Leftrightarrow J^*(x^*)+J(x)=\langle x^*,x \rangle \Leftrightarrow  x \in\partial J^* (x^*),
$$
which means that $\big(\partial J  \big)^{-1}=\partial J^*,$ for every $J\in\Gamma_0(H)$.\\
Given a nonempty closed convex subset $C$ of $H$, those functions corresponding to the {\em indicator} ${\rm I}_C$,
to the {\em support function} $\sigma(C,\cdot)$ of $C$, and to the {\em distance function} $d_C$ from the
set $C$, are defined by
$$
 {\rm I}_C: H\to\R\cup\{+\infty\}  \quad\text{with}\;{\rm I}_C(x)=0\;\text{if}\;x\in C\;\text{and}\;
 {\rm I}_C(x)=+\infty\;\text{if}\;x\not\in C,
$$
$$
 \sigma(C,\cdot) : H \to\R\cup\{+\infty\} \quad\text{with} \; \sigma(C,x^*):=\sup_{x\in C}\langle x^*,x \rangle,
$$
$$
   d_C: H\to\R \quad \text{with}\;d_C(x):=\inf_{y\in C}\|x-y\|.
$$
From the definition of $\sigma(C,\cdot)$, we deduce that $\sigma(C,\cdot)$  coincides with the Legendre-Fenchel conjugate of
${\rm I}_C$, that is, $\sigma(C,\cdot)=({\rm I}_C)^*$.\\
When $J={\rm I}_C$ and $x\in C$, we have
 $$x^*\in \partial {\rm I}_C(x)\mbox{ if and only if }
\langle x^*, y-x \rangle \leq 0, \mbox{ for all }y\in C,$$
 so $\partial {\rm I}_C(x)$ is the set ${\rm N}_C(x)$ of
{\em outward normals} of the convex set $C$ at the point $x\in C$, defined by
$$
{\rm N}_{C}(x)=\{x^*\in H:\langle x^*, y-x\rangle \leq 0, \forall y\in C\}.
$$
We have also,
\begin{equation}\label{supp1}
   x^*\in {\rm N}_{C}(x) \quad\text{if and only if}\quad \sigma(C,x^*)=\langle x^*,x \rangle \;\text{and}\;x\in C.
\end{equation}
It is also clear from the inequality characterization above that
\begin{equation}\label{proj1}
   x-{\rm P}_C(x) \in {\rm N}_{C}\big({\rm P}_C(x)\big),\quad\text{for all}\;x\in H,
\end{equation}
where ${\rm P}_C(y)$ denotes the metric projection onto $C$.\\
It is easy to check that
\begin{equation}\label{normCo}
{\rm N}_C(-x)=-{\rm N}_{-C}(x),\;\;{\rm N}_C(y+z)= {\rm N}_{C-z}(y),
\end{equation}
for any $x\in -C$ and $y,\, z$ such that $y+z\in C$.\\
The  Hausdorff distance between two
subsets $C_{1}$ and $C_{2}$ of $H$, denoted by $d_{H}(C_{1},C_{2})$, is defined by
\begin{equation}
\label{hausd} d_{H}(C_{1},C_{2})=\max\big\{ \sup\limits_{x\in C_{2}}
d_{C_{1}}(x), \sup\limits_{x\in
C_{1}}d_{C_{2}}(x)\big\}.
\end{equation}
%where $d(x,C_{1})=\inf\{\|x-y\|; y \in C_{1}\}.$ \\
%\tcr{Rappeler inverse de $\partial J$, ${\rm N}_C(-x)$, projection ${\rm P}_C$,  indicator function ${\rm I}_C$}
\vskip 2mm
{The following lemma will be useful.}
\begin{lemma} \label{lemsup}Let $C_1$ and $C_2$ be
two subsets of a Hilbert space $H$, and $z\in H$ . Assuming that
$d=d_H(C_1,C_2)<+\infty$ (i.e., $C_1$ and $C_2$ are non-empty),
then we have
\begin{equation}\label{base}\left|\sigma(C_1,z)-\sigma(C_2,z)\right|\leq\|z\|\,d_H(C_1,C_2).\end{equation}\end{lemma}
\begin{proof} {From the definition of the Hausdorff distance, we have
$$\sup_{x\in C_1}\;d_{C_2}(x)\leq d=d_H(C_1,C_2).$$
Hence, $C_1\subset C_2+d\overline{\mathbb{B}}$. On the other hand, we have,
$$\sigma(C_1,z)=\sup_{x\in C_1}\left<z,x\right>\leq\sup_{x\in C_2+d\overline{\mathbb{B}}}\left<z,x\right>=\sup_{x\in C_2}\left<z,x\right>+
d\,\sup_{x\in
\overline{\mathbb{B}}}\left<z,x\right>=\sigma(C_2,z)+d\,\left\|z\right\|.$$
Therefore,
\begin{equation}\label{one}\sigma(C_1,z)-\sigma(C_2,z)\leq\;\|z\|\,d_H(C_1,C_2).\end{equation}
Since $C_1$ and $C_2$ play a symmetric role, we obtain (\ref{base}).
%\begin{equation}\label{two}\sigma(C_2,z)-\sigma(C_1,z)\leq\;\|z\|\,d_H(C_1,C_2).\end{equation}
%We know that
%\begin{equation}\label{three}
%\left|\sigma(C_1,z)-\sigma(C_2,z)\right|=\max\left(\sigma(C_1,z)-\sigma(C_2,z),\;
%\sigma(C_2,z)-\sigma(C_1,z)\right).\end{equation} Combining
%(\ref{one}), (\ref{two}) and (\ref{three}) we obtain (\ref{base}).
}
\end{proof}

We collect below some classical results, that will be useful later, concerning maximal monotone operators (see e.g. \cite{B}).\\
%The following results on maximal monotone operators and convex
%functions are required for the proof of our main results and for
%subsequent discussions.
\begin{lemma}\label{lem1}
\begin{enumerate}
  \item[{\rm (i)}] If $A$ is a maximal monotone operator with bounded domain, then
$A$ is onto. \\
  \item[{\rm (ii)}] Let  $A: H\longrightarrow H$ be a linear, bounded and
symmetric operator satisfying:
$$
\langle Ax,x\rangle\geq 0 , \text {for all } x\in H.
$$
 Then $A=\nabla\varphi$ for the continuous convex function
$\varphi(x)=\frac{1}{2}\langle Ax,x\rangle;$ $\forall x\in H.$
%  Let $A: H\longrightarrow H$ satisfies  \textit{$(\mathcal
%{SP}_{1})$}. Then $A=\nabla\varphi$ with $\varphi(x)=\frac{1}{2}\langle Ax,x\rangle;$
%$\forall x\in H.$\\
  \item[{\rm (iii)}] Let $A$  be a maximal monotone operator and $B$ be a maximal
Lipschitz single-valued operator from $H$ into $H$. Then $A+B$ is
maximal monotone.
\end{enumerate}
\end{lemma}
%\begin{lemma}\label{lem1}[\cite{B}, corollaire 2.2]
%If $A$ is a maximal monotone operator with bounded domain, then
%$A$ is onto.
%
%\end{lemma}
%
%\begin{lemma}\label{lem2}\cite{B}
%Let $A: H\longrightarrow H$ satisfy  \textit{$(\mathcal
%{SP}_{1})$}. Then $A=\nabla\varphi$ for the lower semicontinuous,
%convex and proper function $\varphi(x)=\frac{1}{2}<Ax,x>;$
%$\forall x\in H.$
%
%\end{lemma}

%\begin{lemma}\label{lem3}[\cite{B}, Proposition 1.5]
%Let $\varphi:H \rightarrow ]-\infty,+\infty]$ be lower
%semicontinuous and proper function. If $x_{n}\rightarrow x$ weakly
%in $H,$ then
%$$ \liminf_{n\rightarrow \infty} \varphi(x_{n})\geq\varphi(x).$$
%
%\end{lemma}

%\begin{lemma}\label{lem4}\cite{B}
%Let $A$  be a maximal monotone operator and $B$ be a maximal
%Lipschitz single-valued operator from $H$ into $H$. Then $A+B$ is
%maximal monotone.
%
%\end{lemma}
\vskip 2mm
We end this section with the following lemma on the approximation
of unbounded $C(t).$
\vskip 2mm
\begin{lemma}\label{lem5}{\rm \cite{KM2,Mo}}
Let $C:[0,T]\rightrightarrows H$, $t\mapsto C(t)$ satisfy \textit{$(\mathcal {SP}_{2})$}. Then
there exists an $n_{0}\in \mathbb{N}$ such that for all $n\geq
n_{0}$ we have $C_{n}(t):=C(t)\cap \overline{\mathbb{B}}(0,n)\neq
\emptyset$ for $t\in [0,T]$, and
$$
d_{H}(C_{n}(t),C_{n}(s))\leq
8d_{H}(C(t),C(s)) {\leq  8 \big(v(t)-v(s)\big)}, \mbox { for
all }0\leq s\leq t\leq T.
$$

\end{lemma}
%\vskip 1mm \textcolor{red}{\begin{remark}\label{remc}\normalfont
%In \textit{$(\mathcal {SP}_{2})$} the variation of $C$ is
%absolutely continuous, that is, there exists an absolutely
%continuous function $v$ on $[0,T]$ with $v(0)=0$ such that
%$$ d_{H}(C(t),C(s))\leq |v(t)-v(s)|, \text
%{for all } t,s\in [0,T].
%$$
%Note that, if there an absolutely continuous $a:[0,T]\rightarrow
%\mathbb{R}$ satisfying the above inequality, putting
%$v(t)=\int_{0}^{t}|\dot a(\tau)+\varepsilon|d\tau,$ the function
%$v$ fulfills the same inequality as well as the condition $v(0)=0$
%and $\dot v(t) \geq \omega$ with $\omega:=\varepsilon>0.$
%\end{remark}}
%\vskip 1mm

%\section{Implicit sweeping process}

\section{Main result}
The following theorem establishes the well-posedness  (existence and
uniqueness result) of the evolution problem (\ref{eq1.2}).
\vskip 2mm
\begin{theorem}\label{theo1}
Assume that \textit{$(\mathcal {SP}_{1})$}  and \textit{$(\mathcal
{SP}_{2})$} are satisfied. Then for any initial point $u_{0} \in
H,$ with ${Bu_0\in C(0)}$ there exists a unique
Lipschitz continuous mapping $u: [0,T]\rightarrow H$ satisfying
{\normalfont(\ref{eq1.2})}, i.e.
$$ -\dot u(t) \in {\rm N}_{C(t)}(A\dot u(t)+Bu(t)) \quad\mbox{ a.e. }t\in [0,T],\;\;\;u(0)=u_{0}.$$
\end{theorem}
\begin{proof}
We proceed by discretization of the evolution problem
(\ref{eq1.2}): a sequence of continuous mappings $(u_{n}(.))_{n\in
\N}$ in $C([0, T],H)$ will be defined such that the limit of a
convergent subsequence is a solution
of (\ref{eq1.2}). The
sequence is defined via an implicit algorithm. The proof will be divided into five steps.\\
\noindent  \textbf{Step 1}.Construction of approximants $u^{n}_{i}$. Consider for
each $n\in \mathbb{N}^{*}$ the following partition of the interval
$I:=[0,T]$

$$
t^{n}_{i}:=i\frac{T}{n} \hspace{22mm}\;\quad \text {for} \; 0\leq
i\leq n,
$$
\begin{equation}\label{eq4.1}
I^{n}_{i}:= ]t^{n}_{i},t^{n}_{i+1}] \quad\quad \text {for} \;
0\leq i\leq n-1,
\end{equation}
$$
I^{n}_{0}:= \{t^{n}_{0}\}.\hspace{48mm}
$$
Lemma \ref{lem5} ensures
the existence of $n_{0}\in \mathbb{N}$ such that for all $n\geq
n_{0}$ we have $C_{n}(t):=C(t)\cap \overline{\mathbb{B}}(0,n)\neq
\emptyset$ for $t\in [0,T]$, and
\begin{equation}\label{eq4.2}
d_{H}(C_{n}(t),C_{n}(s))\leq
8d_{H}(C(t),C(s)) {\leq  8 \big(v(t)-v(s)\big)}, \mbox { for
all }0\leq s\leq t\leq T.
\end{equation}
%\textcolor{red}{According to Remak \ref{remc}, we may and do
%suppose (in addition to the equality $v(0)=0$) that $v$ is
%increasing \tcb{nondecreasing?}.\\}
We propose the following numerical method based on the discretization of (\ref{eq1.2}):\\
%\begin{algorithm}\label{algo}\normalfont
Set $\mu_{n}:=\frac{T}{n}.$ Fix $n\geq n_{0}.$ We choose by induction : \\
 $\bullet$ $u_{0}^{n}=u_{0},$ \\
For $i=0,1,\ldots n-1 :$ \\
$\bullet$ {Find} $z^{n}_{i+1}$ by solving the following variational inclusion
\begin{equation}
\label{algo}
 -z^{n}_{i+1} \in
{\rm N}_{C_{n}(t^{n}_{i+1})}(Az^{n}_{i+1}+Bu^{n}_{i}).
\end{equation}
We show later that the following estimation holds:
{\begin{equation} \label{estimation} \| z^{n}_{i+1}
\|\leq \frac{8v(T)}{\beta}\exp(\frac{1}{\beta}\|B\|T):=M.
\end{equation}}
$\bullet$ Set: $u^{n}_{i+1}=u^{n}_{i}+\mu_{n}z^{n}_{i+1}$\\
%$\bullet$ \tcr{$\| z^{n}_{i+1} \|\leq
%\frac{8LT}{\beta}+\frac{1}{\beta}\|B\|\|u_0\|+\exp(\frac{1}{\beta}\|B\|\mu_{n}T):=M.$ Pourquoi mettre ceci dans l'algo ???}
%\textcolor[rgb]{0.00,0.00,1.00}{reponse: pour sweeping process:
%the upper bound pour la derivee de la suite des solutions
%approchees $(\dot u_{n})_{n}$ est important pour avoir les
%convergences. comme $\dot u_{n}(t)=
%\frac{u^{n}_{i+1}-u^{n}_{i}}{\mu_{n}}$ notÈe par la variable
%discrete auxiliere $z^{n}_{i+1}\in C(t^{n}_{i+1})$. \\
%tu souviens dans le papier math program comme le $C$ est borne on
%a obtenu directement cet upper bound. ici comme $C$ est non borne
%il faut passer par ces estimation. metre ceci dans l'algo juste
%pour voir ce calcul important.}
%\end{algorithm}
The numerical method proposed above is well defined. Indeed, for $i=0,$ we have
\begin{equation}\label{eq4.3}
-z^{n}_{1} \in {\rm N}_{C_{n}(t^{n}_{1})}(Az^{n}_{1}+Bu^{n}_{0})
\end{equation}
or equivalently,
\begin{equation}\label{eq4.3bis}
-z^{n}_{1} \in {\rm N}_{C_{n}(t^{n}_{1})-Bu_{0}}(Az^{n}_{1}).
\end{equation}
 Assumption \textit{$(\mathcal {SP}_{1})$} implies that
$A^{-1}:H\longrightarrow H$ is monotone and $\frac{1}{\beta}$
Lipschitz. Hence, (\ref{eq4.3bis}) can be rewritten as
\begin{equation}\label{eq4.4}
0 \in [A^{-1}+{\rm N}_{C_{n}(t^{n}_{1})-Bu_{0}}](Az^{n}_{1}).
\end{equation}
The lemma \ref{lem1} ensures that the operator
$$A^{-1}+{\rm N}_{C_{n}(t^{n}_{1})-Bu_{0}}:C_{n}(t^{n}_{1})-Bu_{0}\rightrightarrows R(A^{-1}+{\rm N}_{C_{n}(t^{n}_{1})-Bu_{0}})$$ is also
maximal monotone with domain $C_{n}(t^{n}_{1})-Bu_{0}.$ \\
As the operator $B$ is bounded and all sets $C_{n}(t)$ are bounded, it
follows from Lemma \ref{lem1} that
$[A^{-1}+{\rm N}_{C_{n}(t^{n}_{1})-Bu_{0}}]$ is onto.\\
Consequently,
$$R(A^{-1}+{\rm N}_{C_{n}(t^{n}_{1})-Bu_{0}})=H,$$
 i.e. there exists
$z^{n}_{1}$ solution of (\ref{eq4.3}) or (\ref{eq4.4}) such that
$Az^{n}_{1}+Bu_{0}\in C_{n}(t^{n}_{1})$. \\
We set then, $u^{n}_{1}=u^{n}_{0}+\mu_{n}z^{n}_{1}$. \\
By (\ref{eq4.3}) we have

\begin{equation}\label{eq4.5}
\langle A z^{n}_{1}+Bu^{n}_{0}-v,z^{n}_{1}\rangle \leq 0, \mbox {
for all}
 \; v\in C_{n}(t^{n}_{1}).
\end{equation}
Using \textit{$(\mathcal {SP}_{1})$} and (\ref{eq4.5}), we have
\begin{eqnarray*}
\beta\| z^{n}_{1} \|^{2} &\leq& \langle A
z^{n}_{1},z^{n}_{1}\rangle
\\
   &= & \langle A
z^{n}_{1}+Bu^{n}_{0}-v+v-Bu^{n}_{0},z^{n}_{1}\rangle\\
   &= & \langle A
z^{n}_{1}+Bu^{n}_{0}-v,z^{n}_{1}\rangle +\langle
v-Bu^{n}_{0},z^{n}_{1}\rangle
   \\
   &\leq & \langle v-Bu^{n}_{0},z^{n}_{1}\rangle
    \\
   &\leq & {\|v-Bu^{n}_{0}\|}\|z^{n}_{1}\|, \mbox{ for all
   } v\in C_{n}(t^{n}_{1}).
\end{eqnarray*}
{{Hence, }
$$\| z^{n}_{1} \| \leq \frac{1}{\beta} \inf_{v\in
C_{n}(t^{n}_{1})}\|v-Bu^{n}_{0}\|=\frac{1}{\beta}d(Bu^{n}_{0},C_{n}(t^{n}_{1})).$$}
{Using the fact that  $Bu^{n}_{0}\in C_{n}(0) $ (since $Bu^{n}_{0}\in C(0)$ and $B$ is bounded), we
get
\begin{eqnarray*}
\| z^{n}_{1} \|
   &\leq & \frac{1}{\beta}d_{H}(C_{n}(t^{n}_{1}),C_{n}(0))
   \\
   &\leq & \frac{8}{\beta}d_{H}(C(t^{n}_{1}),C(0))
    \\
   &\leq & \frac{8}{\beta}\big( v(t^{n}_{1})-v(0)\big)
   \\
   &\leq & \frac{8v(T)}{\beta}.
\end{eqnarray*}}
Now suppose that
$u^{n}_{0},u^{n}_{1},...,u^{n}_{i},z^{n}_{1},z^{n}_{2},...,z^{n}_{i}$
have been constructed. Observe that the operator
$A^{-1}+{\rm N}_{C_{n}(t^{n}_{i+1})-Bu^{n}_{i}}:C_{n}(t^{n}_{i+1})-Bu^{n}_{i}\rightrightarrows
R(A^{-1}+{\rm N}_{C_{n}(t^{n}_{i+1})-Bu^{n}_{i}})$ is maximal monotone
with bounded domain  $C_{n}(t^{n}_{i+1})-Bu^{n}_{i}.$ Then Lemma
\ref{lem1} ensures that
$R(A^{-1}+{\rm N}_{C_{n}(t^{n}_{i+1})-Bu^{n}_{i}})=H.$ Therefore, there
exists $z^{n}_{i+1}$ such that $Az^{n}_{i+1}+Bu^{n}_{i} \in
C_{n}(t^{n}_{i+1})$ solution of
\begin{equation}\label{eq4.6}
-z^{n}_{i+1} \in {\rm N}_{C_{n}(t^{n}_{i+1})}(Az^{n}_{i+1}+Bu^{n}_{i}),
\end{equation}
which allows us to set
$u^{n}_{i+1}:=u^{n}_{i}+\mu_{n}z^{n}_{i+1}$.\\
{Also (\ref{eq4.6}),
$u^{n}_{i}=u^{n}_{0}+\mu_{n}\sum_{k=1}^{i}z^{n}_{k}$,
$Bu^{n}_{0}\in C_{n}(0) $  and \textit{$(\mathcal {SP}_{1})$}
imply that for all $v\in C_{n}(t^{n}_{i+1})$}\\
{\begin{eqnarray*} \| z^{n}_{i+1} \| &\leq&
\frac{1}{\beta}\|v-Bu^{n}_{i}\|=\frac{1}{\beta}\|v-Bu_0-\mu_{n}\sum_{k=1}^{i}Bz^{n}_{k}\|
\\
   &\leq & \frac{1}{\beta}d(Bu_0,C_{n}(t^{n}_{i+1}))+\frac{1}{\beta}\|B\|\mu_{n}\sum_{k=1}^{i}\|z^{n}_{k}\|\\
   &\leq & \frac{1}{\beta}d_{H}(C_{n}(0),C_{n}(t^{n}_{i+1}))+\frac{1}{\beta}\|B\|\mu_{n}\sum_{k=1}^{i}\|z^{n}_{k}\|
   \\
   &\leq & \frac{8}{\beta}d_{H}(C(0),C(t^{n}_{i+1}))++\frac{1}{\beta}\|B\|\mu_{n}\sum_{k=1}^{i}\|z^{n}_{k}\|
    \\
   &\leq &
   \frac{8}{\beta}\big( v(t^{n}_{i+1})-v(0)\big)+\frac{1}{\beta}\|B\|\mu_{n}\sum_{k=1}^{i}\|z^{n}_{k}\|.
\end{eqnarray*}}
Hence, {\begin{eqnarray*} \| z^{n}_{i+1} \|
   &\leq & \frac{8v(T)}{\beta}+\frac{1}{\beta}\|B\|\mu_{n}\sum_{k=1}^{i}\|z^{n}_{k}\|.
\end{eqnarray*}}
Using a discrete version of Gronwall's inequality, we obtain
{$$\| z^{n}_{i+1} \|
  \leq  \frac{8v(T)}{\beta}\exp(\frac{1}{\beta}\|B\|i\mu_{n}
  )$$ or equivalently,
  $$\| z^{n}_{i+1} \|
  \leq  \frac{8v(T)}{\beta}\exp(\frac{1}{\beta}\|B\|T):=M, $$}
    which means that (\ref{estimation}) is satisfied and the numerical method is therefore well defined.\\
\textbf{Step 2}.Construction of the sequence $(u_{n}(.))$.\\
Using the sequences $u^{n}_{i}$ and  $z^{n}_{i}$, we construct the
sequence of mapping $u_{n}:[0,T]\to H,\;t\mapsto u_n(t)$  by defining their
restrictions to each
interval $I^{n}_{i}$ as follows:\\
$$
u_n(t)=\left\{
\begin{array}{lll}
u^{n}_{0} &\mbox{ if }& t=0\\
\\
u_{i}^{n}+\frac{(t-t_{i}^{n})}{\mu_{n}}(u_{i+1}^{n}-u_{
i}^{n})&\mbox{ if } &t \in I^{n}_{i}, i=0,1,\ldots,n-1.
\end{array}
\right.
$$
%for $t=0,$ set $u_{n}(t):=u^{n}_{0},$
%\\ for all $t \in I^{n}_{i} $ $(0\leq i\leq n-1),$ set
%$u_{n}(t)=u_{i}^{n}+\frac{(t-t_{i}^{n})}{\mu_{n}}(u_{i+1}^{n}-u_{
%i}^{n}).$ \\
Clearly the mapping $u_{n}(\cdot)$ is Lipschitz on $[0,T],$ and $M$ is a Lipschitz
constant of $u_{n}(\cdot)$ on $[0,T]$ since for every $t \in
]t_{i}^{n},t_{i+1}^{n}[$, we have
$$\dot u_{n}(t)=\frac{(u_{i+1}^{n}-u_{
i}^{n})}{\mu_{n}}=z^{n}_{i+1}  \mbox{ with }Az^{n}_{i+1}+Bu^{n}_{i}\in
C_{n}(t^{n}_{i+1}). $$
Furthermore, for every $t \in [0,T]$ one has
$u_{n}(t)=u_{0}+\int\limits_{0}^{t} \dot u_{n}(s)ds$. \\
Hence,
$$\|u_{n}(t) \|\leq \|u_{0}\|+MT.$$
By  (\ref{algo}) we have
\begin{equation}
\label{incl}
-z^{n}_{i+1} \in {\rm N}_{C_{n}(t^{n}_{i+1})}(Az^{n}_{i+1}+Bu^{n}_{i}).
\end{equation}
Define the functions $\theta_n$ and $\delta_n$ from $[0,T]$
to $[0,T]$ by $\theta_n(t)=t^{n}_{i+1}$ and
$\delta_n(t)=t^{n}_{i}$ for any $t\in I^{n}_{i}$.
Inclusion (\ref{incl}) becomes
\begin{equation}\label{eq4.7}
- \dot u_{n}(t) \in {\rm N}_{C_{n}(\theta_n(t))}( A\dot
u_{n}(t)+Bu_{n}(\delta_n(t))) \mbox{ a.e. } t\in [0,T].
\end{equation}
\textbf{Step 3}.Convergence of $(u_{n}(.))$. First, We note that $$\sup_{t\in [0,T]}|\theta_n(t)-t|\rightarrow
0 \mbox{ as } n\rightarrow\infty \mbox{ and }\sup_{t\in
[0,T]}|\delta_n(t)-t|\rightarrow 0 \mbox{ as }
n\rightarrow\infty.$$ Now, let us prove the convergence of
sequences $(u_{n})$ and $(\dot u_{n})$. We have for all $n\geq
n_{0}$

\begin{equation*}
\left\{
\begin{array}{c}
 \|u_{n}(t) \|\leq
\|u_{0}\|+TM, \mbox {  } \mbox { for all }  t\in [0,T] \mbox{ and} \\
 \|\dot u_{n}(t) \|\leq M \mbox { for almost all }
  t\in [0,T].
\end{array}
\right.
\end{equation*}
We deduce that the sequence $(u_{n})$ is uniformly bounded in norm and
variation. Using Theorem 0.2.1 in \cite{MM}, there exists a
function $u: [0,T]\rightarrow H$ of bounded variation and a
subsequence, still denoted $(u_{n})$, such that
\begin{equation}
\label{1eq}
u_{n}(t)\rightharpoonup u(t) \mbox { weakly in } H \mbox { for all } t\in [0,T],
\end{equation}
\begin{equation}
\label{2eq}
u_{n}\rightharpoonup u \mbox { in the weak-star topology of } L^{\infty}([0,T],H),
\end{equation}
and, for some $v_*\in L^{2}([0,T],H)$
\begin{equation}\label{eq4.8}
\dot u_{n}\rightharpoonup v_{*} \mbox { in the weak topology of
} L^{2}([0,T],H).
\end{equation}
 In particular, $u(0)=u_{0}$. The Lipschitz continuity of
$u_{n}$ and the weak lower semicontinuity  of the norm give
\begin{eqnarray*}
\|u(t)-u(s)\|\le \liminf_{n\to +\infty}\|u_n(t)-u_n(s)\|\le M |t-s| \mbox { for
all }  t,s\in [0,T],
\end{eqnarray*}
which shows that $u(\cdot)$ is Lipschitz continuous on $[0,T]$, and hence its derivative $\dot{u}(\cdot)$ exists for almost every $t\in [0,T]$.\\
%%%%%%%%%%%%
Fix any $t\in [0,T]$. For each  $w\in H$ with $\|w\|\leq 1$,
 we can write
\begin{align}\label{cvfaible}
   \Big|\langle w,u_n(\theta_n(t))-u(t)\rangle \Big|
  & \leq  \Big|\langle w,u_n(\theta_n(t))-u_n(t)\rangle \Big|+ \Big|\langle w,u_n(t)-u(t)\rangle \Big|\nonumber\\
  & \leq M  \Big|\theta_n(t)-t \Big|+  \Big|\langle w,u_n(t)-u(t)\rangle \Big|.
\end{align}
Taking into account \eqref{1eq}, we get $u_n(\theta_n(t))
\rightharpoonup u(t)$ weakly in $H$ as $n\to\infty$. On the other
hand, we have
$$
   \langle w,u_n(t) \rangle = \langle w,u_0\rangle +\int_0^T\Big\langle \mathbf{1}_{[0,t]}(s)w,\dot{u}_n(s)\Big\rangle\,ds.
$$
Using \eqref{eq4.8} and taking the limit as $n\to\infty$, we obtain
$$
\langle w, u(t)\rangle =\langle w,u_0\rangle +\int_0^T\Big\langle \mathbf{1}_{[0,t]}(s)w,v_*(s)\Big\rangle\,ds
                       = \Big\langle w,u_0+ \int_0^t v_* (s)\,ds \Big\rangle.
$$
The latter equality being true for all $w\in H$, we deduce that $u(t)=u_0+\int_0^t v_* (s)\,ds$,
and this guarantees that $\dot{u}(\cdot)=v_* (\cdot)$ almost everywhere. \\
Consequently,
\begin{equation}\label{eq3.5ter}
   \dot{u}_n(\cdot) \rightharpoonup \dot{u}(\cdot)\:\text {weakly in}\; L^2([0,T],H).
\end{equation}
%according to \eqref{eq3.5bis} again. Furthermore, since $f_n(t)=f(\theta_n(t))$ and $f(\cdot)$ is
%continuous, we have, for every $t\in[0,T]$, that $f_n(t)\to f(t)$ strongly in $H$ as $n\to\infty$.
%%%%%%%%%%
% So by (\ref{eq3.5bis}) $\dot u(t)= v_{*}(t)$ a.e. $t\in[0,T]$.
 Using (\ref{cvfaible}), we deduce that $u_{n}(\theta_{n}(t)) \rightharpoonup u(t)$, as $n\to +\infty$,
weakly in $H$, for all $t\in [0,T]$.\\
\textbf{Step 4}.We show that $u(.)$ is a solution of
(\ref{eq1.2}).\\
 Let us prove first the following viability condition:
\begin{equation}
\label{viab}
 A\dot u(t)+Bu(t)\in C(t),\mbox { a.e. on }[0,T].
\end{equation}
%\textcolor[rgb]{0.00,0.00,1.00}{ici j'ai demontrer cet viabilite
%via deux methode: ad hoc Method et classic method}
%\\
%\textcolor[rgb]{0.00,0.00,1.00}{$\bullet$ \textbf{ad hoc Method}} \\
Fix any $t\in [0,T]$ such that $\dot u(t)$ exists. For each $z\in
H$ with $\|z\| \leq1$, we can write
\begin{eqnarray*}
\langle z,A\dot u(t)+Bu(t) \rangle &=& \langle z,A\dot u(t)+Bu(t)-
A\dot u_{n}(t)-Bu_{n}(\delta_{n}(t))+ A\dot
u_{n}(t)+Bu_{n}(\delta_{n}(t))\rangle
\\
   &= & \langle z,A\dot
u_{n}(t)+Bu_{n}(\delta_{n}(t))\rangle+\langle z,A\dot u(t)- A\dot
u_{n}(t)+Bu(t)-Bu_{n}(\delta_{n}(t)\rangle.
\end{eqnarray*}
as $A\dot u_{n}(t)+ Bu_{n}(\delta_{n}(t))\in
C_{n}(\theta_{n}(t))\subset C(\theta_{n}(t))$, the last inequality
becomes
\begin{eqnarray}\label{star}
\langle z,A\dot u(t)+Bu(t) \rangle &\leq&
\sigma(C(\theta_{n}(t)),z)+\langle z,A\dot u(t)- A\dot
u_{n}(t)+Bu(t)-Bu_{n}(\delta_{n}(t)\rangle.
\end{eqnarray}
From  the property of
{the} support function  {(see Lemma \ref{lemsup})},
we derive
\begin{equation*}
|\ \sigma(C(\theta_{n}(t)),z)-\sigma(C(t)),z)|\ \leq \|z
\|d_{H}(C(\theta_{n}(t)),C(t)).
\end{equation*}
Hence,
\begin{equation}\label{dstar}
\sigma(C(\theta_{n}(t)),z) \leq \sigma(C(t)),z)+|\
{v(\theta_{n}(t))-v(t)}|\
\end{equation}
Combining (\ref{star}) and (\ref{dstar}) we obtain
\begin{eqnarray}\label{integ}
\langle z,A\dot u(t)+Bu(t) \rangle &\leq& \sigma(C(t),z)+|\
 {v(\theta_{n}(t))-v(t)}|\ + \langle z,A\dot u(t)-
A\dot
u_{n}(t)\rangle+\nonumber\\
&&\langle z,Bu(t)-Bu_{n}(\delta_{n}(t)\rangle .
\end{eqnarray}
Itegrating (\ref{integ}) with
$\tau>0$ small, we obtain
\begin{eqnarray*}
\int\limits _{t-\tau}^{t+\tau} \langle z,A\dot u(s)+Bu(s) \rangle
ds &\leq& \int\limits _{t-\tau}^{t+\tau} \sigma(C(s),z)ds +
\int\limits _{t-\tau}^{t+\tau}\langle z,A\dot u(s)- A\dot
u_{n}(s)\rangle ds +\\
&& \int\limits _{t-\tau}^{t+\tau}\langle z,B u(s)-
Bu_{n}(\delta_{n}(s))\rangle ds+  {
\int_{0}^{T}|v(\theta_{n}(t))-v(t)|dt}.
\end{eqnarray*}
 {It is easily seen that
$\int_{0}^{T}|v(\theta_{n}(t))-v(t)|dt\rightarrow0$ as
$n\rightarrow\infty.$} Also the weak convergence of $u_{n}$ and
$\dot u_{n}$ to $u$ and $\dot u$ in $L^{2}([0,T];H)$ respectively
and the properties of $A$ and $B$ yield as $n\rightarrow\infty$
\begin{eqnarray*}
\int\limits _{t-\tau}^{t+\tau} \langle z,A\dot u(s)+Bu(s) \rangle
ds &\leq& \int\limits _{t-\tau}^{t+\tau} \sigma(C(s),z)ds.
\end{eqnarray*}
Dividing by $2\tau$ and letting $\tau$ tend to zero, the Lebesgue
differentiation theorem gives
\begin{eqnarray*}
 \langle z,A\dot u(t)+Bu(t) \rangle
&\leq&  \sigma(C(t),z).
\end{eqnarray*}
The latter inequality being true for all $z\in H$, we deduce
according to the closdeness and the convexity of $C(t)$ that $A\dot
u(t)+Bu(t) \in C(t)$, which means that (\ref{viab}) is proved.\\
Finally we show that $u(.)$ satisfies the differential
inclusion in (\ref{eq1.2}). By (\ref{eq4.4}) and the definition of
the normal cone, we have
\begin{equation}\label{eq4.10}
\langle -\dot u_{n}(t),v-A\dot u_{n}(t)-Bu_{n}(\delta_{n}(t))
\rangle \leq 0,\;\forall
 \; v\in C_{n}(\theta_{n}(t)),\mbox {
a.e. } t\in [0,T].
\end{equation}
We claim that for all $t\in [0,T]$ for which (\ref{eq4.10}) holds
and $v\in C_{n}(t)$ we have

\begin{equation}\label{eq4.11}
\langle -\dot u_{n}(t),v-A\dot u_{n}(t)-Bu_{n}(\delta_{n}(t))
\rangle \leq  {\epsilon_{n}(t)},
\end{equation}
with
{$\epsilon_{n}(t)=8M|v(\theta_{n}(t))-v(t)|$}.
\\ Indeed, by (\ref{eq4.2}) we have  {$ v\in C_{n}(t) \subset
C_{n}(\theta_{n}(t))+8|v(\theta_{n}(t))-v(t)|\mathbb{B}$}. So,
there exists $\tilde{v}\in C_{n}(\theta_{n}(t))$ with
 {$\|v-\tilde{v}\|\leq
8|v(\theta_{n}(t))-v(t)|.$}\\
By (\ref{eq4.10}), we obtain
\begin{eqnarray*}
\langle -\dot u_{n}(t),v-A\dot u_{n}(t)-Bu_{n}(\delta_{n}(t))
\rangle &=& \langle -\dot u_{n}(t),v-\tilde{v}+\tilde{v}-A\dot
u_{n}(t)-Bu_{n}(\delta_{n}(t)) \rangle
\\
   &= & \langle -\dot
u_{n}(t),v-\tilde{v} \rangle+\\
& &\langle -\dot
u_{n}(t),\tilde{v}-A\dot u_{n}(t)-Bu_{n}(\delta_{n}(t)) \rangle\\
   &\leq &  {8M|v(\theta_{n}(t))-v(t)|=\epsilon_{n}(t)},
\end{eqnarray*}
the claim (\ref{eq4.11}) follows. \\
Choose arbitrary $t_{0}\in [0,T]$ and $v_{0}\in
C(t_{0})$. Suppose that $\tau >0$ is such that for all most all $t\in
[0,T]$, there exists a unique projection
$$v(t)=\left\{
\begin{array}{l}
{\rm P}_{C(t)}(v_0)\in C(t),  \mbox{ if } t\in[t_{0}-\tau,t_{0}+\tau],\\
{\rm P}_{C(t)}(A \dot u(t)+Bu(t))=A \dot u(t)+Bu(t)\in C(t),  \mbox{ otherwise. }
\end{array}
\right.
$$
%$$v(t)={\rm P}_{C(t)}(v_0)\in C(t),  if t\in[t_{0}-\tau,t_{0}+\tau]$$ and $$v(t)={\rm P}_{C(t)}(A \dot u(t)+Bu(t))\in C(t),  otherwise.$$ 
Hence, $v(.)$ is a $L^\infty([0,T];H)$
selection of the  {absolutely continuous}
set-valued map $C(.)$ on $[0,T].$ The boundedness
of $v(.)$ on the compact interval $[0,T]$ implies that $v(t)\in C_{n}(t)$ for all $n\in \mathbb{N}$ sufficiently
large. {Using (\ref{eq4.11}), we have}
\begin{equation}\label{eq4.12}
\int\limits _{0}^{T} \langle -\dot
u_{n}(t),v(t)-A\dot u_{n}(t)-Bu_{n}(\delta_{n}(t)) \rangle dt \leq
\int\limits _{0}^{T}
{\epsilon_{n}(t)}dt,
\end{equation}
or equivalently,
\begin{equation}\label{eq4.13}
\int\limits _{0}^{T} \langle \dot u_{n}(t),A\dot
u_{n}(t) \rangle dt +\int\limits _{0}^{T}
\langle \dot u_{n}(t),Bu_{n}(\delta_{n}(t)) \rangle dt +
\int\limits _{0}^{T} \langle \dot u_{n}(t),-v(t)
\rangle dt \leq  {\int\limits _{0}^{T}
\epsilon_{n}(t)}dt.
\end{equation}
From the properties of $A$, we note that the function $x(t)\mapsto
\displaystyle\int\limits _{0}^{T} \langle x(t),Ax(t) \rangle dt $ is convex and
weakly lower semicontinuous on $L^{2}([0,T];H).$ \\
So we have,
\begin{eqnarray}\label{eq4.14}
\int_{0}^{T}\langle  \dot u(t),A\dot u(t)\rangle
dt  \leq \liminf_{n\to \infty} \int\limits
_{0}^{T} \langle \dot u_{n}(t),A\dot
u_{n}(t)\rangle dt,
\end{eqnarray}
and
\begin{eqnarray}\label{eq4.15}
\int_{0}^{T}\langle  \dot u(t),-v(t)\rangle dt =
\lim_{n\to \infty} \int\limits _{0}^{T} \langle
\dot u_{n}(t),-v(t)\rangle dt.
\end{eqnarray}
On the other hand, we have $B=\nabla \varphi_B$, for the continuous
convex function $\varphi_B(x)=\frac{1}{2}\langle Bx,x\rangle$.
Therefore, the absolute continuity of $\varphi_B\circ u$ and
$\varphi_B\circ u_{n}$ gives
\begin{eqnarray}
  \int\limits _{0}^{T}\langle Bu(t),\dot u(t)\rangle \,dt &=&   \int\limits _{0}^{T} \frac{d}{dt} \varphi_B(u(t))dt
  = \varphi_B(u(T))-\varphi_B(u(0)) \nonumber\\
  &\leq&\liminf_{n\to \infty}\Big(\varphi_B(u_n(T))-\varphi_B(u_n(0))\Big) \nonumber\\
  &=&\liminf_{n\to\infty} \Big( \int\limits _{0}^{T} \frac{d}{dt} \varphi_B(u_n(t))\,dt \Big)\nonumber\\
  &=& \liminf_{n\to \infty}\int\limits _{0}^{T} \langle Bu_{n}(t),\dot u_{n}(t)\rangle \,dt, \label{eq4.16}
\end{eqnarray}
where the inequality is due to the weak lower semicontinuity of
$\varphi_B$ on $H$ and to the fact that $u_n(T) \rightharpoonup u(T)$ weakly in
$H$ as $n\to \infty$. \\
Since,
$$
 \int\limits _{0}^{T} \Big\vert \langle Bu_{n}(t)-Bu_{n}(\delta_{n}(t)),\dot u_{n}(t)\rangle \Big\vert\,dt \leq M^{2}
   \|B\|\int_0^T|t-\delta_{n}(t)\vert\,dt,
 $$
we deduce that
\begin{equation}\label{eq4.17}
 \liminf_{n\to \infty}\int\limits _{0}^{T} \langle Bu_{n}(t),\dot u_{n}(t)\rangle \,dt
 =\liminf_{n\to \infty}\int\limits _{0}^{T} \langle Bu_{n}(\delta_{n}(t)),\dot{u}_{n}(t)\rangle\, dt.
\end{equation}
%
% \tcr{Je ne comprend pas où utilise-t-on ceci ?}\\
%\textcolor[rgb]{0.00,0.00,1.00}{ ceci est la justification de (32), il suffit d'estimer la
% difference}.\\
{Since} $\int_{0}^{T}\epsilon_{n}(t)dt\rightarrow0$ as
$n\rightarrow\infty,$ inequalities
(\ref{eq4.13}),(\ref{eq4.14}), (\ref{eq4.15}), (\ref{eq4.16}) and
(\ref{eq4.17}) yield as $n\rightarrow\infty$
\begin{equation}\label{eq4.18}
\int\limits _{0}^{T} \langle -\dot
u(t),v(t)-A\dot u(t)-Bu(t) \rangle dt \leq 0.
\end{equation}
Using the definition of $v(t)$ above, we get
\begin{equation}\label{eq4.18b}
\int\limits _{t_0-\tau}^{t_0+\tau} \langle -\dot
u(t),v(t)-A\dot u(t)-Bu(t) \rangle dt \leq 0.
\end{equation}
Dividing (\ref{eq4.18b})  by $2\tau$, letting $\tau$ goes to zero and using the
Lebesgue differentiation theorem, we get
\begin{equation}\label{eq4.19}
 \langle -\dot
u(t_{0}),v(t_{0})-A\dot u(t_{0})-B u(t_{0}) \rangle  \leq 0,
\end{equation}
or equivalently,
\begin{equation}\label{eq4.20}
 \langle -\dot
u(t_{0}),v_{0}-A\dot u(t_{0})-B u(t_{0}) \rangle  \leq 0,
\end{equation}
for all $t_{0}\in [0,T]$, outside a fixed set of measure zero
$\{t^{n}_{i},i=0,1,...,n   ; n \in \mathbb{N}\}$, and all $v_{0}
\in C(t_{0})$. This means that $u(.)$ is a solution of the
inclusion (\ref{eq1.2}).
\vskip 2mm\noindent
 \textbf{Step 5}.Uniqueness of the solution. Suppose that $(u_{1},u_{2})$ are two
 solutions satisfying (\ref{eq1.2}) such that $u_{1}(0)=u_{2}(0)=u_{0}.$
 Then for almost every $t\in [0,T],$ we have  for
 $i=1,2$
\begin{equation}\label{eq4.19b}
\langle -\dot u_{i}(t),v-A\dot u_{i}(t)-B u_{i}(t) \rangle \leq 0,
\mbox { for all}
 \; v\in C(t).
\end{equation}
Using the fact that $A\dot u_{i}(t)+B u_{i}(t)\in C(t)$ a.e., we
obtain, for a.e. $t\in [0,T],$

\begin{equation*}
\left\{
\begin{array}{c}
 \langle \dot u_{1}(t),A\dot u_{1}(t)+B u_{1}(t)-A\dot u_{2}(t)-B u_{2}(t)\rangle \leq 0, \\
  \langle -\dot u_{2}(t),A\dot u_{1}(t)+Bu_{1}(t)-A\dot u_{2}(t)-B u_{2}(t) \rangle \leq
  0.
\end{array}
\right.
\end{equation*}
By adding the last two inequalities, we get
\begin{equation*}
  \Big\langle \dot u_{1}(t)-\dot u_{2}(t),A\big( \dot u_{1}(t)-\dot u_{2}(t)\big)+B\big( u_{1}(t)- u_{2}(t)\big)\Big\rangle \leq 0.
\end{equation*}
Since $A$ is coercive, we obtain
\begin{equation*}
 \beta \|\dot u_{1}(t)-\dot u_{2}(t)\|^{2}\leq \|B\| \|\dot u_{1}(t)-\dot u_{2}(t)\|\|u_{1}(t)- u_{2}(t)\|.
\end{equation*}
As $u_{1}(0)=u_{2}(0)=u_{0}$, we get
\begin{equation*}
 \|\dot u_{1}(t)-\dot u_{2}(t)\|\leq \frac{\|B\|}{ \beta}\int\limits _{0}^{t} \| \dot u_{1}(\tau)- \dot u_{2}(\tau)\|d\tau,
\end{equation*}
 which means by Gronwall's inequality that $\dot u_{1}(t)=\dot u_{2}(t)$ for a.e. $t\in
[0,T].$  Therefore $ u_{1}(t)=u_{2}(t)$ for all $t\in [0,T].$
%Pourquoi ?? pour tout t et non pas a.e. t?? \\
%\textcolor[rgb]{0.00,0.00,1.00}{puisque les solutions sont
%absoluments continues er donc pour tout $t\in [0,T]$ on a
%$u_{1}(t)=u_{1}(0)+\int\limits _{0}^{t} \dot u_{1}(s)ds
%=u_{2}(0)+\int\limits _{0}^{t} \dot u_{2}(s)ds=u_{2}(t)$}
%\\
 The proof of Theorem \ref{theo1} is thereby completed.
\end{proof}

\section{Application to quasistatic frictional contact problem  }
As an application of the sweeping process problem (\ref{eq1.2}), we consider the following evolution variational inequality
\begin{eqnarray}\label{eq1.3}
\left\{
\begin{array}{l}
  \text{Find} \; u:[0,T]\longrightarrow H   \mbox{ such that } \; \dot u(t)=\frac{du(t)}{dt} \in \mathcal{K}\; \text {a.e.} \; t\in [0,T]\;\text {and } \hspace{30mm}\\
  \\
 a\big(\dot u(t),v-\dot u(t)\big)+b\big(u(t),v-\dot u(t)\big)+j(v)-j(\dot u(t))\geq \langle f(t),v-\dot u(t)\rangle,\;\forall
 \; v\in \mathcal{K}.\\ \\
  u(0)=u_{0}\in H,
\end{array}
\right.
\end{eqnarray}
Assume that the following assumptions are satisfied:\\
\\ \textit{$(\mathcal {VI}_{1})$} $ \mathcal{K}\subset H$ is a nonempty,
closed and convex cone (hence $0\in \mathcal{K}).$
\\ \textit{$(\mathcal {VI}_{2})$} $a(\cdot,\cdot)$, $b(\cdot,\cdot):H\times H\to\R$ are two real
continuous bilinear and symmetric forms satisfying  for all $u\in H$
$$
a(u,u)\geq \alpha_{0} \|u\|^{2},
$$
for some positive constant $\alpha_{0} >0$ and $ b(u,u)\geq 0.$
\\ \textit{$(\mathcal {VI}_{3})$} $j:\mathcal{K} \longrightarrow \mathbb{R}$
is a convex, positively homogeneous of degree $1$ (i.e. $j(\lambda x)=\lambda j(x),\;\forall \lambda >0$) and Lipschitz continuous with $j(0)=0.$
\\ \textit{$(\mathcal {VI}_{4})$} $f\in W^{1,1}([0,T];H)$ with $b(u_0,v)+j(v)\geq \langle f(0),v \rangle,\; \text {for all }
 \; v\in \mathcal{K}$.
 \vskip 2mm\noindent
 \begin{remark}\label{rem2}\normalfont
%%(i) We can suppose that $f\in W^{1,2}([0,T];H)$.  In fact, the regularity assumption on $f$ will have an effect on the way the moving set $C(t)=f(t)-C$  of the corresponding sweeping process is Lipschitz continuous or absolute continuous (See Remark \ref{rem2b}). For simplicity, we suppose that $f\in C^{1}([0,T];H)$.\\
{ The compatibility condition on the initial data
$$b(u_0,v)+j(v)\geq \langle f(0),v \rangle, \; \forall v\in \mathcal{K},$$
ensures that initially the state is in equilibrium and that $Bu_0\in C(0)$ (see (\ref{comp})).
}
\end{remark}
 %We note that compatibility condition
 %(or $l(0)=0$ with $j$ \tcr{non-negative???}) \\
%\textcolor[rgb]{0.00,0.00,1.00}{ comme $C(t)=f(t)-C$ et $l(0)=0$
%pour que $0 \in C=\partial J(0)=\{\xi\in H:\langle\xi, v\rangle
%\leq J(v),\forall v\in H\}$ il faut $J(v) \geq 0, \forall v \in H.
%$}
\vskip 1mm
The evolution variational inequality  (\ref{eq1.3}) is of great interest in the modeling of the quasistatic frictional contact problems (see
\cite{DL,GLT,SM}). In a mechanical language, the bilinear form $a(\cdot,\cdot)$ represents the viscosity term, the bilinear form $b(\cdot,\cdot)$ represents the elasticity term, the functional $j$ represents the friction functional of Tresca type.\\
For our purpose of motivation, the main concern is to prove that the variational inequality (\ref{eq1.3}) is of type
(\ref{eq1.2}). In other words, we will convert the quasistatic variational inequality (\ref{eq1.3}) to the
problem of finding a solution {of the} sweeping process (\ref{eq1.2}).\\
Let us first extend the function $j$ from $\mathcal{K}$ to the
whole space $H$ by introducing the functional $J: H\to
\R\cup\{+\infty\}$, $z\mapsto J(z)$
defined by
\begin{equation}\label{extend}
J(z)=\left\{
\begin{array}{c}
 j(z),\quad z\in \mathcal{K} ,\\
  +\infty,\quad z\notin \mathcal{K}.
\end{array}
\right.
\end{equation}
Since $\mathcal{K}$ is a nonempty, closed and convex cone, and $j$
is convex, positively homogeneous of degree $1$  and Lipschitz
continuous on $\mathcal{K}$, we deduce that the extended
functional $J:H \longrightarrow \mathbb{R}\cup
\{+\infty\}$ is proper, positively
homogeneous of degree $1$,
convex and lower semicontinuous with $J(0)=0.$\\
With this extension, (\ref{eq1.3}) is equivalent to
\begin{eqnarray}\label{eq3.1}
\left\{
\begin{array}{l}
  \text{Find} \; u:[0,T]\longrightarrow H   \mbox{ such that  for a.e. } t\in [0,T]\;\text {we have } \hspace{30mm}\\
  \\
a(\dot u(t),v-\dot u(t))+b(u(t),v-\dot u(t))+J(v)-J(\dot u(t))\geq \langle f(t),v-\dot u(t)\rangle\;\forall v\in H.\\ \\
  u(0)=u_{0}\in H,
\end{array}
\right.
\end{eqnarray}
%\begin{equation}\label{eq3.1}
%\left\{
%\begin{array}{c}
%  \text{Find} \; u:[0,T]\longrightarrow H,  u(0)=u_{0}\in H , \hspace{50mm}\\
% \text{such that } \text {for a.e.} \; t\in [0,T]\;\text {we have } \hspace{50mm}\\
% a(\dot u(t),v-\dot u(t))+b(u(t),v-\dot u(t))+J(v)-J(\dot u(t))\geq \langle f(t),v-\dot u(t)\rangle\; \text {for all}
% \; v\in H.
%\end{array}
%\right.
%\end{equation}
%(ii) \underline{Implicit differential inclusion governed by subdifferential} \\
Let $A$ and $B$ be the linear bounded and symmetric operators
associated respectively  to the bilinear forms $a(\cdot,\cdot)$ and $b(\cdot,\cdot)$, that is,
\begin{equation}\label{AB}
\langle Au,v\rangle=a(u,v) \mbox{ and } \langle Bu,v\rangle=b(u,v), \mbox{ for
all }u,v \in H.
\end{equation}
Using the definition of the subdifferential given in (\ref{subdiff}), we can rewrite (\ref{eq3.1}) in the following form
\begin{equation}\label{eq3.2}
\left\{
\begin{array}{c}
 f(t)-A\dot u(t)-Bu(t) \in \partial J(\dot u(t)) \quad\text{a.e.\;}t\in [0,T],\\
  u(0)=u_{0}\in H.\hspace{30mm}
\end{array}
\right.
\end{equation}
The following Proposition shows the equivalence between the variant of {the} sweeping process introduced in (\ref{eq1.2}) and the quasistatic variational inequality (\ref{eq3.1}).
\vskip 2mm\noindent
%(see Proposition \ref{prop1}).
%In this section, we prove that (\ref{eq1.3}) may be converted to a
%sweeping process problem, as is shown by the following Proposition
%\ref{prop1} below. \\
%First, we give the definition of solutions.
%\begin{definition}\label{def1}
%
% Let $u_{0}\in H$ be given. A function $u\in
%W^{1,2}([0,T];H)$ is called a solution of (\ref{eq1.3}) (resp. of
%(\ref{eq1.2})) if $\dot u(t) \in \mathcal{K}\; \text {a.e.} \;
%t\in [0,T]$ (resp. $A\dot u(t)+Bu(t) \in C(t)\; \text {a.e.} \;
%t\in
%[0,T]$) and if (\ref{eq1.3}) is satisfied (resp. (\ref{eq1.2}) is satisfied). \\
%\end{definition}
\begin{proposition}\label{prop1}
Assume that assumptions \textit{$(\mathcal {VI}_{1})$}-\textit{$(\mathcal
{VI}_{4})$} are satisfied.  The function $u:[0,T]\longrightarrow H$ is a
solution of  {\normalfont (\ref{eq1.3})} if and only if it is a solution of  the sweeping process
 {\normalfont(\ref{eq1.2})}, where $A$ and $B$ are the linear bounded and
symmetric operators associated with $a(\cdot,\cdot)$ and
$b(\cdot,\cdot)$, and
$C(t):=f(t)-\partial J(0),\;t\in [0,T]$ with $J$ defined in  {\normalfont(\ref{extend})}.
\end{proposition}
\vskip 1mm
\begin{proof} It is easy to see that $u$ is a solution of the variational inequality (\ref{eq1.3}) if and only if it is a solution of the differential inclusion (\ref{eq3.2}). \\
From the properties of the subdifferential of $\partial J$ and since $J(0)=0$, we deduce that the subset
\begin{equation}\label{setC}
C:=\partial J(0)=\{\xi\in H:\langle\xi,
v\rangle \leq J(v),\forall v\in H\},
\end{equation}
 is a closed convex
subset in $H$.\\
%Using the fact that for every $x,z\in H$, we have
%\begin{equation*}
%  J'(x;z)=\max\big\{\langle x^*, z\rangle : x^*\in \partial J(x)\big\}=\sigma(\partial J(x),z),
%\end{equation*}
%where
%$$
%J'(x;z):=\inf_{\tau>0}\tau^{-1}\big(J(x+\tau z)-J(x)\big)
%      = \lim_{\tau\downarrow 0}\tau^{-1}\big(J(x+\tau z)-J(x) \big).
%      $$
%In particular for $x=0$, we get
%\begin{equation}\label{Jz}
%J'(0;z)=\max\big\{\langle x^*, z\rangle : x^*\in \partial J(0)\big\}=\sigma(\partial J(0),z).
%\end{equation}
Since $J$ is positively homogeneous of degree $1$ with $J(0)=0$, from a standard result in convex analysis, we have
%we have $J'(0;\cdot )=J(\cdot)$. Using the definition of the set
%$C$ in (\ref{setC}), we derive from (\ref{Jz}) that
$$J(z)=\sigma(C,z)={\rm I}^*_C(z).$$
Hence,
$$\partial J(\cdot )=\partial {\rm I}^*_C(\cdot) \mbox{ and } J^*(\cdot)={\rm I}^{**}_C(\cdot)={\rm I}_C(\cdot).$$
On the other hand, we have
$$p\in \partial J(z) \Longleftrightarrow z\in \partial J^*(p).$$
Therefore,
\begin{equation}\label{equiv}
p\in \partial J(z) \Longleftrightarrow z\in \partial {\rm I}_C(p) \Longleftrightarrow z\in {\rm N}_C(p), \mbox{ with } C=\partial J(0).
\end{equation}
Applying (\ref{equiv}) to (\ref{eq3.2}) and using (\ref{normCo}), we get for $\text{a.e.\;}t\in [0,T]$,
\begin{align*}
   f(t)-A\dot u(t)-Bu(t) \in \partial J(\dot u(t))     & \Longleftrightarrow \dot u(t) \in
{\rm N}_{C}\big(f(t)-A\dot u(t)-Bu(t)\big)\\
& \Longleftrightarrow     \dot u(t) \in {\rm N}_{C-f(t)}(-A\dot u(t)-Bu(t))   \\
& \Longleftrightarrow -\dot u(t) \in {\rm N}_{C(t)}(A\dot u(t)+Bu(t)), \\
\end{align*}
  with $C(t):=f(t)-C=f(t)-\partial J(0)$.   \\
%In terms of $C$, by standard convex-analysis
%calculus \cite{ET}, we have $J(z)=\sigma_{C}(z)=\delta^{*}(z)$,
%where $\delta_{C}(.)$ denoting the indicator function of the set
%$C$ and $\delta_{C}^{*}(.)$ is its Legendre-Fenchel conjugate.
Hence, problem (\ref{eq3.2}) is equivalent to
\begin{equation}\label{eq3.3}
\left\{
\begin{array}{c}
 -\dot u(t) \in {\rm N}_{C(t)}(A\dot u(t)+Bu(t)) \quad\text{a.e.\;}t\in [0,T],\\
  u(0)=u_{0}\in H,\hspace{30mm}
\end{array}
\right.
\end{equation}
which is exactly of the form of the variant of the sweeping process introduced in (\ref{eq1.3}).
\end{proof}
\vskip 1mm
As a consequence of Theorem \ref{theo1}, we have the following existence and uniqueness result for the quasistatic variational inequality
 (\ref{eq1.3}).
\vskip 1mm
\begin{corollary}\label{cor1}
Assume that assumptions \textit{$(\mathcal {VI}_{1})$}-\textit{$(\mathcal
{VI}_{4})$} are satisfied. Then for each $u_0\in H$, the evolution variational inequality  {\normalfont(\ref{eq1.3})} has a unique solution $u$.
\end{corollary}
\vskip 1mm
\begin{proof} Let us check that all assumptions of Theorem \ref{theo1} are satisfied. It is clear that assumptions \textit{$(\mathcal {VI}_{2})$} are equivalent to
 \textit{$(\mathcal {SP}_{1})$}. Let us check now that  \textit{$(\mathcal {SP}_{2})$} is verified. For every $t\in [0,T]$, we have $C(t)=f(t)-C=f(t)-\partial J(0)$. It is clear that $C(t)$ is a closed and convex set of $H$. On the hand, we have,
{\begin{eqnarray}\label{comp}
 Bu_0\in C(0) &\Longleftrightarrow& f(0)-Bu_0\in C \nonumber\\
 &\Longleftrightarrow& f(0)-Bu_0\in \partial J(0) \nonumber\\
 &\Longleftrightarrow& \langle f(0),v\rangle \leq j(v)+b(u_0,v),\;\forall v\in \mathcal{K}.
 \end{eqnarray}}
 Let us show now that the set-valued map $t\mapsto C(t)$ moves in an absolute continuous way. In fact, for all $0\leq s \leq t \leq T$, we have
 \begin{eqnarray*}
d_{H}(C(t),C(s))&\leq& \big\|f(t)-f(s)\big\| \\
   &= & \Big\|\int\limits _s^t \dot f(\tau) d\tau\Big\|
   \\
   &\leq & \int\limits _s^t \| \dot f(\tau) \|d\tau\\
   &= &
{v(t)-v(s) \mbox{ with }} v(t):=\int\limits _0^t \| \dot f(\tau)
\|d\tau.
\end{eqnarray*}
which means that $C(.)$ varies in an
{absolutely} continuous way.\\
 Hence, all assumptions of Theorem \ref{theo1} are satisfied. The existence and uniqueness of a solution to problem (\ref{eq1.3})  is simply a consequence of Theorem \ref{theo1}.
\end{proof}
%\vskip 1mm \textcolor{red} { on supprime la remarque ci-dessous}
%\begin{remark}\label{rem2b}\normalfont
%We note that Proposition \ref{prop1} is also valid if $C(.)$
%varies in the absolutely continuous way, in this case we have just to take
%$f\in W^{1,2}([0,T];H)$. Then $t\mapsto C(t)$ is assumed to be
%Lipschitz continuous just to simplify the proof of the main
%result.
%\end{remark}
\vskip 1mm
\begin{example} \normalfont{\bf Quasistatic frictional contact problem involving viscoelastic materials with short memory \cite{SM}.} Let $\Omega \subset {\R}^2$ be the section of a tube with infinity length  $\Omega \times ]-\infty,
+\infty[$ (see Figure \ref{fig1}). We
assume that $\Omega$ is an open bounded connected set with a
regular boundary $\Gamma=\partial \Omega$, ($\Gamma$ is a
one-dimensional manifold of class $C^m\; (\:m \geq 1\:)$ and
$\Omega$ is located on one side of $\partial \Omega$). We suppose
that $\Gamma$ is composed of three parts $\Gamma
=\overline{\Gamma}_1\cup \overline{\Gamma}_2\cup
\overline{\Gamma}_3$, with $\Gamma_i$ three open subsets for
$i=1,2,3$ and meas$(\Gamma_1)>0$. In this case the well-known
Korn's inequality is satisfied. We assume that the cylinder is
clamped on $\Gamma_1$ and in contact with a rigid foundation on
$\Gamma_3$. The cylinder deforms under the action of a surface
density force $f_0$ on $\Omega$ acting in the
axle-direction and traction forces of density $f_2$ on $\Gamma_2$ (for more details about the mathematical modeling of the antiplane shear, we refer to \cite{SM}, Chapter 8).
For simplification, we will omit the dependence of functions with respect to the space variable $x\in \Omega\cup\Gamma$ and the time variable $t\in [0,T]$. Moreover, the dot represents the time derivative i.e. $\dot{u}=\frac{du}{dt}$.\\
The
displacement of the
cylinder is governed by the following quasistatic variational inequality:
\begin{eqnarray}\label{memb}
\left\{
\begin{array}{l}
  \text{Find} \; u:[0,T]\longrightarrow H   \mbox{ such that  for a.e. } t\in [0,T]\;\text {we have } \hspace{30mm}\\
  \\
a(\dot u(t),v-\dot u(t))+b(u(t),v-\dot u(t))+J(v)-J(\dot u(t))\geq \langle f(t),v-\dot u(t)\rangle\;\forall v\in H,\\ \\
  u(0)=u_{0}\in H,
\end{array}
\right.
\end{eqnarray}
where $H=\{v\in H^1(\Omega) : v=0
\mbox{ on } \Gamma_1   \}$, the bilinear forms:
$a(\cdot,\cdot),\;b(\cdot,\cdot)\;:\;H\times H\to
\R,\;(u,v)\mapsto a(u,v),\;b(u,v)$, the frictional functional
$J:H\to \R,\;v\mapsto J(v)$ and the function $f:[0,T]\to
H,\;t\mapsto f(t)$ are defined respectively by
\begin{subequations}
\label{bilin}
\begin{empheq}{align}
&  a(u,v)=\int_{\Omega}\eta\: \nabla u\cdot \nabla v\, dx  \label{bb}\\
  &b(u,v)=\int_{\Omega}\kappa\: \nabla u\cdot \nabla v\, dx \label{ba}\\
  &J(v)=\int_{\Gamma_3}g\vert v\vert d\Gamma.\\
  & \langle f(t),v\rangle=\int_{\Omega} f_0(t)v\,dx+\int_{\Gamma_2}f_2(t) v\,d\Gamma.
  \end{empheq}
\end{subequations}
%The variational formulation of problem (P)\textcolor{red} {P ou 4.11}is given by (specify the VI).\\
\begin{figure}
\begin{center}
\includegraphics[width=3in]{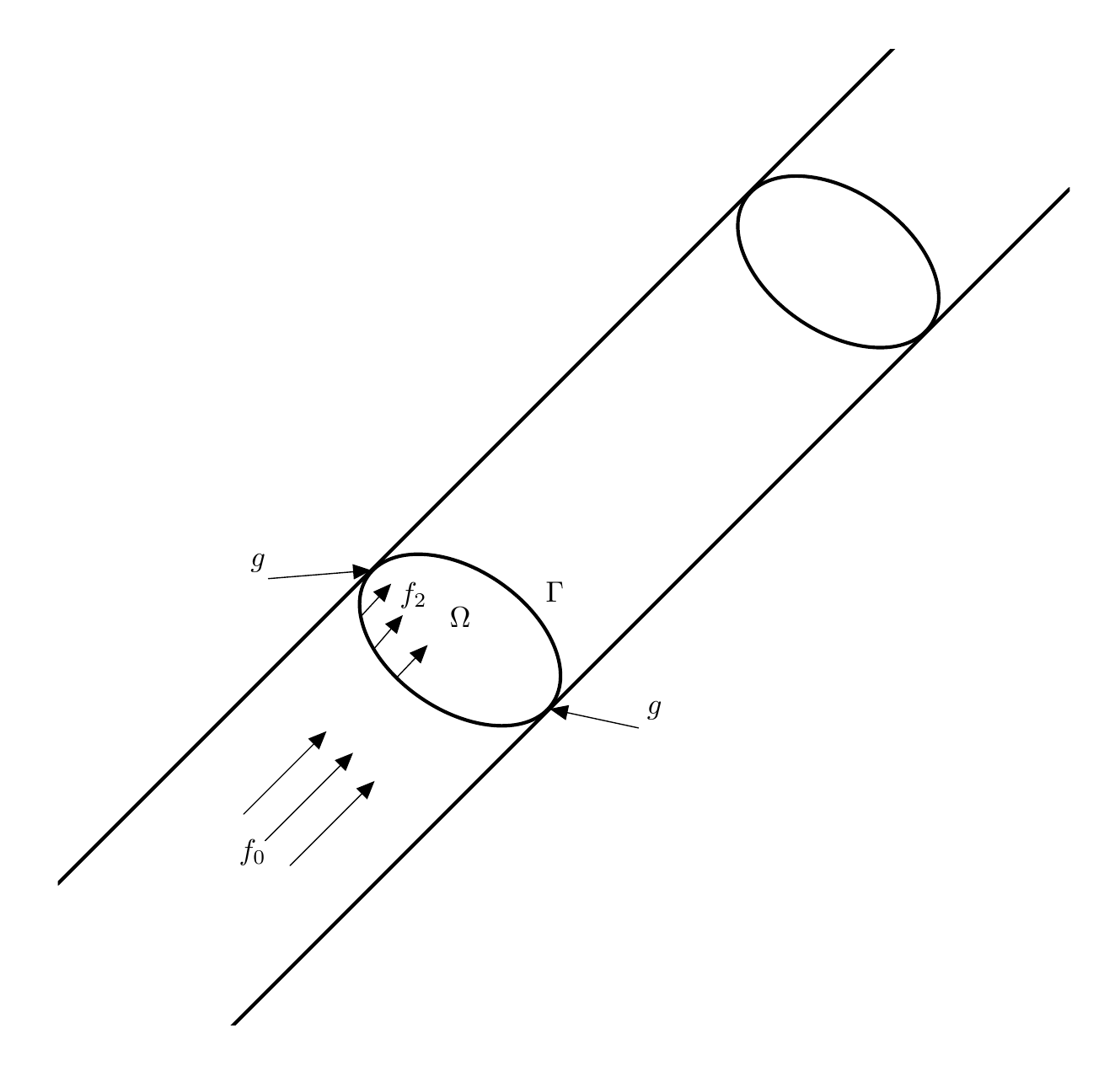}\includegraphics[width=2in]{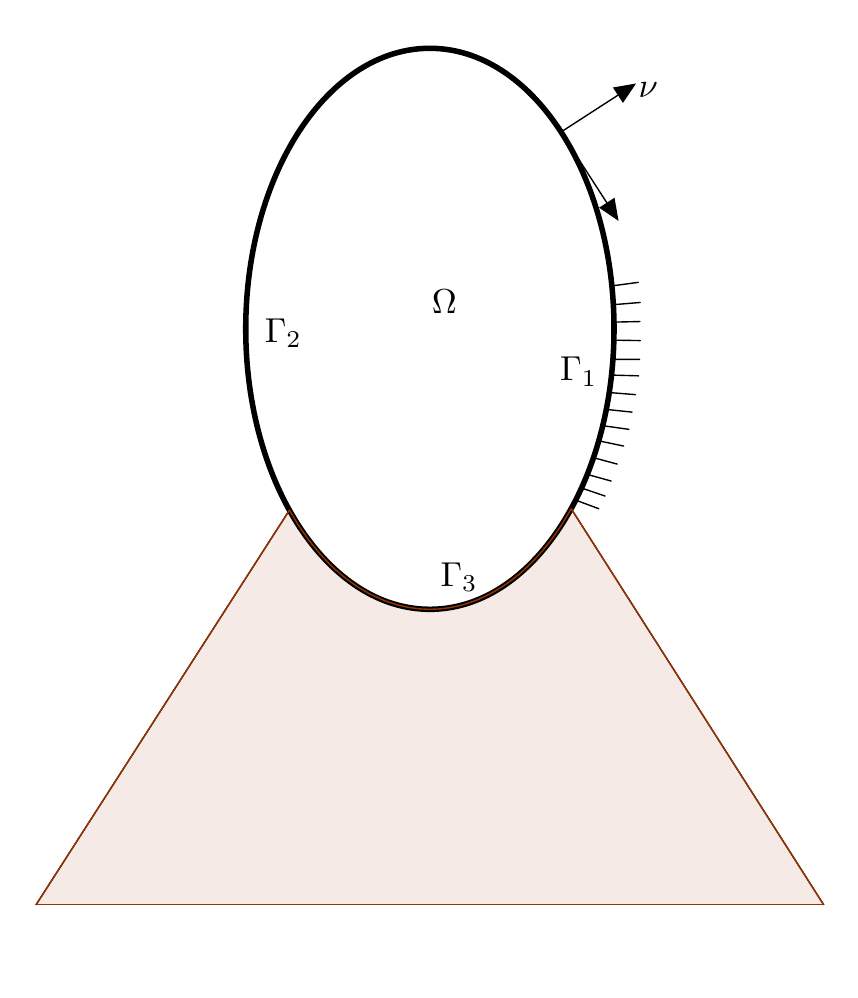}
\caption{Cross section of the cylinder in contact with a foundation.}
\label{fig1}
\end{center}
\end{figure}
%\begin{figure}
%\begin{center}
%\includegraphics[width=2in]{figures/membrane2D.pdf}
%\caption{ }
%\label{fig2}
%\end{center}
%\end{figure}
\begin{remark}\normalfont
The variational formulation (\ref{memb}) derived from the following problem:\\
Find a displacement field $u: [0,T]\times \Omega \to \R,\;(t,x)\mapsto u(t,x)$ such that
\begin{subequations}
\label{stat}
\begin{empheq}[left={({\mathcal P})}\empheqlbrace]{align}
  &{\rm div}\Big(\eta (x)\nabla \dot{u}(t,x)+ \kappa(x)\nabla u(t,x) \Big)+f_0=0 &\mbox{ in } ]0,T[\times\Omega; \label{stata}\\
  & u(t,\cdot)=0&\mbox{ on }]0,T[\times\Gamma_1  \label{statb};&\\
  & \eta\partial_{\nu} \dot{u}+ \kappa\partial_{\nu}u=f_2 &\mbox{ on }]0,T[\times\Gamma_2 \label{statc};\\
  & \big\vert  \eta\partial_{\nu} \dot{u}+ \kappa\partial_{\nu}u \big\vert \leq g & \mbox{ on }]0,T[\times\Gamma_3 \label{statd};\\
  &  \eta\partial_{\nu} \dot{u}+ \kappa\partial_{\nu}u = -g\frac{\dot{u}}{\vert \dot{u}\vert}\mbox{ if } \dot{u}\neq 0, &\mbox{ on }]0,T[\times\Gamma_3 \label{state};\\
&  u(0,\cdot)=u_0(\cdot)& \mbox{in } \Omega.\label{statf}
  \end{empheq}
\end{subequations}
Here $\nu$ denotes the unit outer normal on the boundary $\Gamma$.
Equation (\ref{stata}) is the equilibrium state equation where a
viscoelastic constitutive law with short memory is assumed,
(\ref{statb}) is the Dirichlet boundary condition on $\Gamma_1$,
(\ref{statc}) is the traction boundary condition on $\Gamma_2$,
(\ref{statd})-(\ref{state}) are the frictional conditions and
(\ref{statf}) is the initial condition. For more details we refer
to \cite{SM} page 191. If the contact is modeled with a nonmonotone normal compliance condition and a unilateral constraint, then it is possible to study the problem in the framwork of variational-hemivariational inequalities (see e.g. the recent papers \cite{bs, hms} and references therein).
\end{remark}
\vskip 1mm\noindent
We suppose that the viscosity coefficient $\eta $, the Lam\'e coefficient $\kappa $, the forces $f_0$, $f_2$ and the friction function $g$
satisfy the following conditions
\begin{subequations}
\label{lame}
\begin{empheq}{align}
  & \kappa\in L^{\infty}(\Omega) \label{la}\\
  & \eta\in  L^{\infty}(\Omega) \mbox{ with } \eta (x)\geq \eta ^* \mbox{ a.e. } x\in \Omega \mbox{ (for some } \eta ^*>0).\label{lb}\\
  &f_0\in
{W^{1,1}}([0,T];L^2(\Omega)),\;f_2\in
{W^{1,1}}([0,T];L^2(\Gamma_2))\label{lc}\\
  & g(x)\geq 0 \mbox{ a.e. } x\in \Gamma_3 \mbox{ and } g\in L^2(\Gamma_3).\label{ld}\\
  &
{b(u_0,v)}+\int_{\Gamma_3}g\vert v\vert\: d\Gamma \geq
\int_{\Omega} f_0(0)v\:dx+\int_{\Gamma_2}f_2(0) v\:
d\Gamma,\;\;\forall v\in H.\label{le}
  \end{empheq}
\end{subequations}
As a direct consequence of Corollary \ref{cor1}, we show that problem (\ref{memb}) is well-posed.
\vskip 1mm
\begin{corollary} Assume {\normalfont{(\ref{la})-(\ref{ld})}}. Then for each $u_0\in H$ satisfying {\normalfont (\ref{le})}, problem  {\normalfont (\ref{memb})-(\ref{bilin})} has a unique solution.
\end{corollary}
\vskip 1mm
\begin{proof}We have
$$\vert a(u,v)\vert \leq \Vert \eta\Vert_{\infty} \Vert u\Vert \Vert v\Vert \mbox{ and } \vert b(u,v)\vert \leq \Vert \kappa\Vert_{\infty} \Vert u\Vert \Vert v\Vert.$$
The coercivity of $a(\cdot,\cdot)$ follows from (\ref{lb})
$$a(v,v)\geq \eta ^* \Vert v\Vert^2,\;\forall v\in H.$$
Assumption (\ref{le}) implies the compatibility condition \textit{$(\mathcal {VI}_{4})$}.
All assumptions \textit{$(\mathcal {VI}_{1})$}-\textit{$(\mathcal
{VI}_{4})$}  are satisfied. The conclusion follows by Corollary \ref{cor1}.
\end{proof}
\vskip 5mm
\begin{remark} \normalfont
We note that the existence of a unique solution to  problem  {\normalfont (\ref{memb})-(\ref{bilin})} was obtained in
 \cite{SM} without the compatibility condition (\ref{le}). This condition was used in \cite{SM} for the the study of quasistatic frictional problems with
 elastic materials (see Section 9.3 page 184 and (11.37) page 208 in \cite{SM}). We note that the compatibility condition (\ref{le}) is necessary in many quasistatic problems, it guarantees that the initial state is in equilibrium otherwise the inertial terms $\ddot{u}(t)$ cannot be neglected and the problem is no longer quasistatic (it will be a dynamic of second-order).  For the implicit sweeping process studied in this paper, condition (\ref{le}) is equivalent to the viability condition $Bu_0\in C(0)$ (necessary to start the algorithm since outside this set the normal cone would be empty).
\end{remark}
\end{example}
\bibliographystyle{amsplain}

\begin{thebibliography}{200}

\bibitem{AHT} \textsc{S. Adly, T. Haddad and L. Thibault}, {\em Convex sweeping process in the framework of measure differential inclusions
and evolution variational inequalities}, Math. Program. Ser. B 148
(2014), 5-47.

\bibitem{bs} \textsc{K. Bartosz and M. Sofonea}, {\em The Rothe method for variational-hemivariational inequalities with applications to contact mechanics}.
 SIAM J. Math. Anal. 48 (2016), no. 2, 861-883.
\bibitem{B}
\textsc{H. Brezis}, {\em Op\'erateurs Maximaux Monotones}, North Holland Publ.
Company, Amsterdam- London, (1973)


\bibitem{CV}
\textsc{C. Castaing and M. Valadier}, {\em Convex analysis and measurable
multifunctions}. Springer-Verlag, Berlin (1977)

\bibitem{bounkhel} \textsc{M. Bounkhel}, {\em Existence and uniqueness of some variants of nonconvex sweeping processes}. J. Nonlinear Convex Anal. 8 (2007), no. 2, 311?323.

\bibitem{BT} \textsc{M. Bounkhel and L. Thibault}, {\em Nonconvex sweeping process and
prox-regularity in Hilbert space}. J. Nonlinear Convex Anal. 6,
(2005), 359-374.

\bibitem{DL}
\textsc{D. Duvaut and J. L. Lions}, {\em Inequalities in mechanics and physics}.
Springer-Verlag, Berlin (1976)

%\bibitem{ET}
%I. Ekeland and R. Temam. {\em Convex Analysis and Variational Problems}.
%SIAM, Philadelphia, 1999.
\bibitem{GLT}
\textsc{R. Glowinski, J.L. Lions and R. Tr\'emoli\`ere}, {\em Numerical Analysis of
Variational Inequalities}. North-Holland, Amsterdam, 1981

%\bibitem{Ha1} T.  Haddad, {\em Nonconvex Differential Variational
%Inequality and State-Dependent Sweeping Process}, Journal of
%Optimization Theory an Applications, volume 159, issue 2, (2013),
%386-398.
\bibitem{hms}
\textsc{W. Han, S. Mig\'orski and M.  Sofonea}, {\em  A class of
variational-hemivariational inequalities with applications to
frictional contact problems}. SIAM J. Math. Anal. 46 (2014), no.
6, 3891-3912.
\bibitem{KM2}
\textsc{M. Kunze and M. D. P. Monteiro Marques}, {\em On discretization of
degenerate sweeping process}. Portugalliae Mathematica. 55, 219-232
(1998)

\bibitem{SM} \textsc{M. Sofonea and A. Matei}, {\em Variational inequalities with applications. A study of antiplane frictional contact problems}.
 Advances in Mechanics and Mathematics, 18. Springer, New York, 2009.

\bibitem{MM} \textsc{M. D. P. Monteiro Marques}, {\em Differential inclusions in
nonsmooths mechanical problems, Shokcks and dry Friction}. Progress
in Nonlinear Differential Equations an Their Applications,
Birkhauser. 9 (1993)

\bibitem{M} \textsc{J. J. Moreau}, {\em Evolution problem associated with a moving convex
set in a Hilbert space}, J. Diff. Eqs. 26, (1977),
347-374.

\bibitem{Moreau1}
\textsc{J.J. Moreau}, \textit{Sur l'\'evolution d'un syst\`eme \'elastoplastique}, C. R. Acad. Sci. Paris S\'er. A-B,
273 (1971), A118-A121.

\bibitem{Moreau2}
\textsc{J.J. Moreau}, \textit{Rafle par un convexe variable I}, S\'em. Anal. Convexe Montpellier (1971), Expos\'e 15.

\bibitem{Moreau3}
\textsc{J.J. Moreau}, \textit{Rafle par un convexe variable II,} S\'em. Anal. Convexe Montpellier (1972), Expos\'e 3.

\bibitem{Moreau4}
\textsc{J.J. Moreau}, \textit{On unilateral constraints, friction and plasticity}, in ''New Variational Techniques in Mathematical
Physics'' (G. Capriz and G. Stampacchia, Ed), 173-322, C.I.M.E. II Ciclo 1973, Edizioni Cremonese,
Roma, 1974.

\bibitem{Mo} \textsc{J. J. Moreau}, {\em Intersection of moving convex sets in a normed
space},  Math. Scand., 36 (1975), 159-173.
%\vskip3mm
%\

\end{thebibliography}

\end{document}